\documentclass{amsart}
\usepackage{amssymb}
\usepackage{booktabs,url,longtable}
\usepackage{graphicx,color}
\newtheorem{theorem}{Theorem}[section] 
\newtheorem{lemma}[theorem]{Lemma}
\newtheorem{proposition}[theorem]{Proposition}
\newtheorem{corollary}[theorem]{Corollary}

\theoremstyle{definition}

\theoremstyle{remark}

\numberwithin{equation}{section}

\title{Heron Quadrilaterals Via Elliptic Curves}

\author[F. Izadi]{Farzali Izadi} 
\address{Department of Mathematics,  Faculty of Science, Urmia University, P. O. Box 165,  Urmia 5715799313, Iran}
\email{f.izadi@urmia.ac.ir}

\author[F. Khoshnam]{Foad Khoshnam}
\address{Department of Mathematics, Azarbaijan Shahid Madani University, Tabriz 53751-71379, Iran}
\email{khoshnam@azaruniv.edu}

\author[D. Moody]{Dustin Moody}
\address{Computer Security Division, National Institute of Standards \& Technology, Gaithersburg, MD 20899-8930}
\email{dustin.moody@nist.gov}

\date{}


\begin{document}

\dedicatory{}

\begin{abstract}
A Heron quadrilateral is a cyclic quadrilateral whose area and
side lengths are rational. In this work, we establish a correspondence between
Heron quadrilaterals and a family of elliptic curves of the form $y^2=x^3+\alpha x^2-n^2 x$.  This correspondence generalizes the notions of Goins and Maddox who established a similar connection between Heron triangles and elliptic curves.  
We further study this family of elliptic curves, looking at their torsion groups and ranks.  We also explore their connection with congruent numbers, which are the $\alpha=0$ case.  Congruent numbers are positive integers which are the area of a right triangle with rational side lengths.  
\end{abstract}

\maketitle


\section{Introduction}
A positive integer $n$ is a congruent number if it is equal to the area of a right triangle
with rational sides. Equivalently, $n$ is congruent if the elliptic curve $E_n : y^2=x^3-n^2 x$ has positive rank.  Congruent numbers have been intensively studied, see for example \cite{CAH}, \cite{COA}, \cite{Spe3}. The curves $E_n$ are closely connected with the problem of classifying areas of right rational triangles. Indeed, Koblitz \cite{Koblitz}, used the areas of rational triangles as a motivation for studying elliptic curves and modular forms.  In \cite{GM}, Goins and Maddox generalized  some of Koblitz's notions (\cite[Ch. 1, \S 2, ex. 3]{Koblitz}) by exploring the correspondence
between positive integers $n$ associated with arbitrary triangles (with rational side lengths) which have area $n$ and the family of elliptic
curves $y^2 = x(x - n\tau)(x + n\tau^{-1})$ for nonzero rational $\tau$. Congruent number curves are of course the $\tau=1$ case.

In this work, we extend these ideas to
show a correspondence between cyclic quadrilaterals with rational side lengths and area $n$ (Heron quadrilaterals) and a family of elliptic curves of the form $y^2=x^3+\alpha x^2-n^2 x$.  We give explicit formulas which show how to construct the elliptic curve and some non-trivial points on the curve given the side lengths and area of the quadrilateral.  If we set one of the side lengths to zero, then the formulas collapse to exactly those of Maddox and Goins.  We also show the other direction of the correspondence, that is, how to find a cyclic quadrilateral which corresponds to a given elliptic curve in our family.  We call the pair $(\alpha, n)$ a \emph{generalized congruent number pair} if the elliptic curve $y^2=x^3+\alpha x^2-n^2 x$ has a point of infinite order. We similarly call the curve a generalized congruent number elliptic curve.  The generalized congruent number curves with $\alpha=0$ are precisely the congruent number curves.  Stated in this way, our results relate generalized congruent number pairs with cyclic quadrilaterals with area $n$.

We also study the family of curves defined by the generalized congruent number pairs, looking at their torsion groups and ranks.  The torsion groups are usually ${\mathcal T}=\mathbb Z/2\mathbb Z$, although in some cases it is $\mathbb Z/2\mathbb Z\times \mathbb Z/2\mathbb Z, \mathbb Z/2\mathbb Z\times \mathbb Z/4\mathbb Z$ or $\mathbb Z/6\mathbb Z$.  Studying families of elliptic curves with torsion group $\mathbb{Z}/2\mathbb{Z}$ with high rank has been of much interest \cite{ARC},\cite{Cam},\cite{Kr},\cite{Kr2},\cite{Nag}. The highest known rank for a curve with ${\mathcal T}=\mathbb Z/2\mathbb Z$ is 19, due to Elkies \cite{Duj3}.  Fermigier found infinite families with rank at least 8 \cite{F1,F2,F3}.

Any elliptic curve with a 2-torsion point may be written in the form $E_{\alpha,\beta}:y^2=x^3+\alpha x^2+\beta x$.
Special cases of the family of the curves $E_{0,\beta}: y^2=x^3+\beta x$, and their ranks have been studied by many authors including Bremner and Cassels \cite{Br}, Kudo and Motose \cite{Ko}, Maenishi \cite{Ma}, Ono and Ono \cite{On}, Izadi, Khoshnam and Nabardi \cite{IK}, Aguirre and Peral \cite{AP}, Spearman \cite{Spe,Spe2}, and Hollier, Spearman and Yang \cite{Ho}. The general case was studied by Aguirre, Castaneda, and Peral \cite{Ag} and they found curves of rank 12 and 13.  See \cite{Duj3,Duj4} for tables with the highest known ranks for other fixed torsion groups, including references to the papers where each curve can be found.

In this work, the curves we study are of the form $E_{\alpha,-n^2}$.  We believe this is the first time in the literature curves of this form have been examined.  We find many such curves with rank (at least) 10.  We also construct an infinite family of the $E_{\alpha,-n^2}$ with rank at least 5.  All of these curves arise from cyclic quadrilaterals.  Furthermore, in the special case with $\alpha=0$, we find infinite families of congruent number curves with ranks 2 and 3, matching the results of \cite{Rub}, \cite{Wat}, \cite{Spe3}.

This work is organized as follows.  In section 2 we review basic facts about cyclic quadrilaterals.  Section 3 details the correspondence between cyclic quadrilaterals and the elliptic curves, and includes our main result.  We examine the torsion groups of the family of elliptic curves we study in Section 4.  In sections 5 and 6 we find examples of congruent number curves with high rank, as well as high rank curves from the family $E_{\alpha,-n^2}$.  We conclude with some examples and data in Section 7.

\section{Cyclic quadrilaterals}
A cyclic polygon is one with vertices upon which a circle can be circumscribed. Specifically, we will focus on cyclic quadrilaterals.
Mathematicians have long been interested in  cyclic quadrilaterals. For example, consider Kummer's complex construction to generate
Heron quadrilaterals outlined in \cite{Dic}.
The existence and parametrization of quadrilaterals with rational side lengths (and additional conditions) has a long history \cite{A-N,Dai,Dic,Gup,I-V}.
Buchholz and Macdougall \cite{Buc-Mac1} have shown that there exist no nontrivial
Heron quadrilaterals having the property that the rational side lengths form an arithmetic or geometric progression.   In \cite{MK}, (cyclic) Brahmagupta quadrilaterals were used to construct infinite families of elliptic curves with torsion group $\mathbb Z/2\mathbb Z\times \mathbb Z/2\mathbb Z$
having ranks 4, 5, and 6.

A convex quadrilateral is cyclic if and only if its opposite angles are supplementary.  An example of a quadrilateral which is not cyclic is a non-square rhombus.  Another characterization of cyclic quadrilaterals can be given by Ptolemy's theorem:  if the diagonals have lengths $p,q$, then a convex quadrilateral is cyclic if and only if $pq=ac+bd$.  Given four side lengths such that the sum of any three sides is greater than the remaining side, there exists a cyclic quadrilateral with these side lengths \cite{CMG},\cite{Pet}.  The area of a cyclic quadrilateral with side lengths $a,b,c,d$ can be found using Brahmagupta's formula
$$\sqrt{(s-a)(s-b)(s-c)(s-d)},$$
where $s=(a+b+c+d)/2$.  Letting $d=0$, the formula collapses to Heron's formula for the area of a triangle.  It is known that a cyclic quadrilateral has maximal area among all quadrilaterals with the same side lengths.

%

Assume we have a cyclic quadrilateral whose consecutive sides have lengths $a,b,c,$ and $d$, with rational area $n$.  Let $\theta$ be the angle between the sides with lengths $a$ and $b$.  Then  by considering the Law of Cosines and the area formula, we have
\begin{equation} \label{loc} \cos \theta =\frac{a^2+b^2-c^2-d^2}{2(ab+cd)}\qquad \mbox{and} \qquad \sin \theta=\frac{2n}{ab+cd}. \end{equation}

\section{Cyclic quadrilaterals and elliptic curves}
In this section we establish a correspondence between cyclic quadrilaterals whose area and side lengths are rational and elliptic curves. In \cite{GM}, the authors created a similar correspondence between triangles with rational area and elliptic curves, which in some sense were generalizations of congruent number curves.  We follow the same initial approach.

We use the notation from the previous section.  From equation \eqref{loc}, we see that both $\cos \theta$ and $\sin \theta$ are rational.  Set $\tau$ to be
$$\tau=\tan \frac{\theta}{2}=\frac{\sin \theta}{1+\cos \theta}=\frac{4n}{(a+b)^2-(c-d)^2}.$$
Note that
$$\tau+\tau^{-1}=\frac{ab+cd}{n}$$
and
$$\tau-\tau^{-1}=-\frac{a^2+b^2-c^2-d^2}{2n}.$$
From \eqref{loc}, consider that
$$a^2-2ab \cos \theta +b^2 =c^2+2cd \cos \theta +d^2,$$
so then
$$(a-b \cos \theta)^2+(b^2-d^2)\sin^2 \theta=(c+d \cos \theta)^2.$$
Thus, if we set 
$u=a-b\cos \theta$, $v=b\sin \theta$, and $w=c+d\cos \theta$,
then 
$$u^2+(1-d^2/b^2)v^2=w^2.$$  
Hence we know there exists a $t$ such that
$$u=(1+d/b)t^2-(1-d/b),$$ 
$$v=2t,$$
$$w=(1+d/b)t^2+(1-d/b),$$
in terms of
$$t=\frac{b}{b+d}\frac{u+w}{v}=\frac{b}{b+d}\frac{a+c-(b-d)\cos\theta}{b\sin\theta}$$
$$=\frac{(a+c)^2-(b-d)^2}{4n}.$$


Set
$$x_1=nt=\frac{(a+c)^2-(b-d)^2}{4},$$
$$y_1=ax_1=a\frac{(a+c)^2-(b-d)^2}{4},$$
The point $P_1=(x_1,y_1)$ is on the curve $y^2=x^3+\alpha x^2+\beta x$ where

$$\alpha=\frac{2n}{\tan\theta}+d^2=\frac{a^2+b^2-c^2+d^2}{2} \qquad \mbox{and} \qquad \beta=-n^2.$$
We denote this defined  cubic equation by $E_{\alpha,-n^2}.$  The discriminant of the curve is
$$\Delta(E_{\alpha,-n^2})=n^4(a^2b^2+a^2d^2+b^2d^2+2abcd)$$
and is nonzero because $a,b,c$ and $d$ are positive. Hence, the cubic does indeed define a nonsingular curve.
A point $P=(x,y)$ has order 2 if and only if $y=0$; hence as $n\neq 0$ then $y_1=(a/4)(a+b+c-d)(a-b+c+d) \neq 0$ and so $P_1$ does not have order 2.
This construction generalizes Goins and Maddox's technique since setting $d=0$, the formulas for $\tau, t, \alpha, \beta, \sin \theta, \cos \theta, (x_1,y_1)$ obtained are exactly those in \cite{GM}.

We can easily find other points on the elliptic curve $E_{\alpha,-n^2}.$  Namely, let
$$P_2=(x_2,y_2)=\left(-\frac{(a+d)^2-(b-c)^2}{4},b\frac{(a+d)^2-(b-c)^2}{4}\right),$$
$$P_3=(x_3,y_3)=\left(-\frac{(a+b)^2-(c-d)^2}{4},d\frac{(a+b)^2-(c-d)^2}{4}\right).$$

Note that $a=y_1/x_1, b=-y_2/x_2, d=-y_3/x_3$.  It can be checked that $(x_1,y_1)+(x_2,y_2)+(x_3,y_3)=\infty$.  Furthermore, using the height pairing matrix it can be checked that any two of the set of points $(x_1,y_1), (x_2,y_2), (x_3,y_3)$ are linearly independent.  So the rank of $E$ is generically at least 2.

Using the addition law, we also have
\begin{equation} \label{addzero} (x,y)+(0,0)=\left(\frac{-n^2}{x},\frac{n^2y}{x^2} \right),\end{equation}
\begin{equation} \label{double} [2](x,y)=\left(\frac{(x^2+n^2)^2}{4y^2},\frac{(x^2+n^2)(x^4+2\alpha x^3-6n^2x^2-2\alpha n^2 x+n^4)}{8y^3} \right).\end{equation}
In particular
$$x_{2P_1}=\frac{(ac+bd)^2}{4a^2} ,$$
$$x_{2P_2}=\frac{(ad+bc)^2}{4b^2} ,$$
$$x_{2P_3}=\frac{(ab+cd)^2}{4d^2} ,$$
from which we can derive the following
\begin{equation}\label{x1} \frac{x_1^2+n^2}{x_1}=ac+bd,\end{equation}
\begin{equation}\label{x2} \frac{x_2^2+n^2}{x_2}={-(}ad+bc{)}, \end{equation}
$$\frac{x_3^2+n^2}{x_3}={-(}ab+cd{)}.$$
Obviously, using the above quantities, we could solve for $c, \cos \theta, \sin \theta, \tau$, etc.  In particular
$$\sin A=\frac{2n}{ab+cd}=-\frac{2nx_3}{x_3^2+n^2},$$
$$\sin B=\frac{2n}{ad+bc}=-\frac{2nx_2}{x_2^2+n^2},$$
and
$$\cos A=\frac{a^2+b^2-c^2-d^2}{2(ab+cd)}=\frac{x_3(d^2-\alpha)}{x_3^2+n^2}=\frac{x_3^2-n^2}{x_3^2+n^2},$$
$$\cos B=\frac{-a^2+b^2+c^2-d^2}{2(ad+bc)}=\frac{x_2(\alpha-b^2)}{x_2^2+n^2}=\frac{n^2-x_2^2}{x_2^2+n^2}.$$

The correspondence we have illustrated is the main result of this work.  We note that for every rational value $n$, there are many cyclic quadrilaterals with area $n$.  For example, any rectangle with side lengths $k$ and $n/k$.  Still, we find the above correspondence very interesting.

\begin{theorem}
\label{main}
For every cyclic quadrilateral with rational side lengths and area $n$, there is an elliptic curve
$$E_{\alpha,-n^2}:y^2=x^3+\alpha x^2-n^2x$$
with 2 rational points, neither of which has order 2.  Conversely, given an elliptic curve $E_{\alpha,-n^2}$ with positive rank, then there is a cyclic quadrilateral with area $n$ whose side lenghts are rational (under the correspondence given above).
\end{theorem}
\begin{proof}
Given a cyclic quadrilateral, the above construction shows how to construct $E_{\alpha,-n^2}$, with points $P_1$ and $P_2$.  The points $P_i$ are of order 2 if and only if $y_{P_i}=0$, which implies $x_{P_i}=0$ as well, since for example $y_1=ax_1$.  However the $x$-coordinates are all non-zero, since in any quadrilateral (cyclic or not) the largest side length is less than the sum of the other three sides.  This shows one direction.

For the converse, we fix a point $P_1=(x_1,y_1)$ of infinite order (which exists since $E_{\alpha,-n^2}$ has positive rank).  We may replace $P_1$ by $-P_1$ or $P_1+(0,0)$ (see \eqref{addzero}) so that we can take $x_1, y_1>0$.  Set $a=y_1/x_1$.  Then consider the three equations:
\begin{equation} \label{cst1} \frac{a^2+b^2-c^2+d^2}{2}=\alpha,\end{equation}
\begin{equation} \label{cst2} \frac{1}{16}(a-b-c-d)(a-b+c+d)(a+b-c+d)(a+b+c-d)=-n^2,\end{equation}
$$ac+bd=\frac{x_1^2+n^2}{x_1}.$$
We set $\zeta=\frac{x_1^2+n^2}{x_1}$, which can be easily shown to be equal to $a^2-\alpha +2n^2/x_1.$
Let
$$h(x):=(b^2-a^2)x^2-2\zeta bx+(2\alpha a^2+\zeta^2-a^4-a^2b^2),$$
and let $r$ be a root of $h(x)=0$.  It can be checked that a solution to the three equations above is given by
$c=(-br+\zeta)/a,$ and $d=r.$  In order for $d=r$ to be rational, the discriminant of $h$ must be a square, equivalently
$$C(b,z):b^4-2\alpha b^2 +(\zeta^2-a^4+2a^2 \alpha)=z^2.$$
We can then express
$$c=\frac{a\zeta\pm bz}{a^2-b^2},$$
$$d=-\frac{b\zeta\pm az}{a^2-b^2}.$$
This quartic curve $C(b,z)$ is actually birationally equivalent to $y^2=x^3+\alpha x^2-n^2 x$.

\begin{lemma}
The curve $C(b,z)$ is birationally isomorphic to the curve $E_{\alpha,-n^2}:y^2=x^3+\alpha x^2-n^2x$.
\end{lemma}
\begin{proof}
Note the curve $C(b,z)$ has rational point $(-a,\zeta)$.  
Under the transformation $f_1(b,z) \to (b+a,z)$, we map to the curve
$$C_1(b,z):b^4-4ab^3+(6a^2-2\alpha)b^2+4a(\alpha -a^2)b+\zeta^2,$$
with rational point $(0,\zeta)$.
We now map to a Weierstrass curve by
$$
\begin{aligned}
f_2(b,z)=(x,y)=\Big (&-\frac{2}{3b^2}((3a^2-\alpha)b^2+6a(\alpha-a^2)b+3\zeta(z+\zeta)),
\\ &\frac{4}{b^3}(-a\zeta b^3+\zeta(3a^2-\alpha)b^2+a(z+3\zeta)(\alpha-a^2)b+\zeta^2(z+\zeta)\Big ),
\end{aligned}
$$
$$C_2(x,y):y^2=x^3+(-4/3 \alpha^2 -8\alpha a^2+4a^4-4\zeta^2)x-16/27\alpha(\alpha^2-18\alpha a^2+9a^4-9\zeta^2).$$
We now perform a simple linear change of variables $f_3(x,y)=(\frac{1}{4}(x-4\alpha/3),\frac{1}{8}y)$, sending the curve to
$$C_3(x,y):y^2=x^3+\alpha x^2+\frac{1}{4}(\alpha^2-2a^2 \alpha+a^4-\zeta^2)x.$$
We compute
$$\begin{aligned}
\alpha^2-2a^2 \alpha+a^4-\zeta^2&=\alpha^2-2a^2\alpha+a^4-\Big(a^2-\alpha+2\frac{n^2}{x_1}\Big)^2  \\
&=4n^2\frac{-a^2x_1+\alpha x_1-n^2}{x_1^2}\\
&=4n^2\frac{-y_1^2+\alpha x_1^2-n^2x_1}{x_1^3}\\
&=-4n^2.
\end{aligned}$$
Thus, $C_3(x,y)$ is really just $y^2=x^3+\alpha x^2-n^2x$.  Composing the maps $f_1, f_2,$ and $f_3$, we see the curves are birationally equivalent.
\end{proof}

\noindent We continue the proof of Theorem \ref{main}. Since $E_{\alpha,-n^2}$ has positive rank, then so does $C(b,z)$.  In other words, there are infinitely many rational points $(b,z)$ on $C(b,z)$.
Given any rational values for $(b,z)$, we can then compute $c$ and $d$.  It remains to check that $c$ and $d$ are positive.

We have an infinite number of choices for $b$.  We will pick a ``small" $b$ so that the quantity $a^2-b^2$ will be positive.  We then want $a\zeta-bz>0$ and $-b\zeta+az>0$, which will guarantee $c,d>0$.
Note that when $b=a$, both the line $z=(\zeta/a)b$ and the hyperbola $z=(a\zeta)/b$ intersect the curve $C(b,z)$ at the same point $(a,\zeta)$.  Looking in the first quadrant (i.e., where $b,z>0$), the line and hyperbola can possibly intersect $C(b,z)$ additional times.
The line will intersect $C(b,z)$ if $b^2=(\zeta^2+2a^2\alpha-a^4)/a^2$.  The hyperbola will intersect $C(b,z)$ if $b^4+(a^2-2\alpha)b^2+\zeta^2=0$.  Let $r^*$ be the minimum positive $b$-value of any intersections of $C(b,z)$ with either the line or hyperbola.  Also observe that for positive $b$ near $b=0$, the value of the hyperbola $z=(a\zeta)/b$ goes to infinity.

We claim the quartic curve $C(b,z)$ intersects the $z$-axis.  We must check that $(0, \sqrt{\zeta^2-a^4+2a^2\alpha})$ is a real point on $C(b,z)$, i.e., that $\zeta^2-a^4+2a^2\alpha>0$.  Equivalently, this is $\alpha^2 x_1^2+4n^2(n^2+a^2 x_1-\alpha x_1))>0$.  Now, since $a^2x_1^2=x_1^3+\alpha x_1^2-n^2 x_1$ then $x_1^2=-\alpha x_1+n^2+a^2 x_1$.  Substituting this in to the previous equation we have $\alpha^2 x_1^2+4n^2 x_1^2$, and thus $\zeta^2-a^4+2a^2\alpha > 0$.

We now utilize the above analysis to show how to choose a ``small" $b$.  Since the curve $C(b,z)$ has positive rank, then it has an infinite number of rational points.  The curve is obviously symmetric about the $b$-axis, and it is easy to check that it does not intersect the $b$-axis.  Thus the number of connected components (over $\mathbb{R}$) is one.   In particular, we can conclude there are an infinite number of rational points with $z>0$.  By the density of rational points on positive rank curves (see Th. 5 Ch. 11 of \cite{Sk}), we can choose a rational $b_0$ with $0<b_0<\epsilon <a$ (for any $\epsilon< \mbox{min}\{a, \sqrt{\zeta^2+2a^2\alpha-a^4}/a , r^*\}$) yielding a rational point on $C(b_0,z_0)$.  As seen in the analysis above, the point $(0, \sqrt{\zeta^2-a^4+2a^2\alpha})$  lies beneath the hyperbola $z=(a\zeta)/b$ and above the line $z=(\zeta/a) b$, hence the same is true for $(b_0,z_0)$.  Thus with this choice of $(b_0,z_0)$ (for small enough $\epsilon$), we see that both $c$ and $d$ are positive.

In fact, this argument shows we have an infinite number of possibilities for positive $b,c,$ and $d$.  The cyclic quadrilateral with side lengths $(a,b,c,d)$ will then correspond to $E_{\alpha,-n^2}$ since the equations \eqref{cst1} and \eqref{cst2} are satisfied.  This completes the proof.
\end{proof}

We remark that while the proof showing the existence of a cyclic quadrilateral is not completely constructive (since we require $c,d>0$), in practice it is not hard to produce the cyclic quadrilaterals.  Start with two points $P_1$ and $P_2, $ which are not of order 2 (and whose sum is $P_1+P_2$ is also not of order 2).  Write $P_i=(x_i,y_i)$, and set $a=y_1/x_1$, $b=-y_2/x_2$.
Then set $P_3=-P_1-P_2=(x_3,y_3)$, and $d=-y_3/x_3$.  By assumption, $P_3$ is not 2-torsion, and hence $y_3 \neq 0$, so $d \neq 0$.   By replacing $P_i$ by $P_i+(0,0)$ or $-P_i$, we can assume $x_1>0, y_1>0$, $x_2<0, y_2>0$, and $x_3<0, y_3>0$ so that $a,b,$ and $d$ are positive.
Then compute $c=(x_1+n^2/x_1-bd)/a$.
If we assume the rank of $E_{\alpha,-n^2}$ is positive, then we will have an infinite number of choices for $P_1$ and $P_2$.  From numerical experiments, we have observed that $c$, as derived above, is usually positive. However, in the case $c$ is negative, we can simply replace $P_1$ and/or $P_2$ until $c>0$.
%
%

If neither of the points $P_1,P_2$ has infinite order, then the cyclic quadrilateral must be of a special form.

\begin{theorem}
\label{rank0}
If the curve $E_{\alpha,-n^2}$ arising from a cyclic quadrilateral has rank 0 then the associated quadrilateral is either a square, or an isosceles trapezoid with three sides equal ($a=b=d$) such that $(a+c)(3a-c)$ is a square.   The torsion group is $\mathbb{Z}/6\mathbb{Z}$ for this rank 0 case.
\end{theorem}
\begin{proof}
We will show in the next section that the torsion group is $\mathbb{Z}/2\mathbb{Z}$, $\mathbb{Z}/2\mathbb{Z} \times \mathbb{Z}/2\mathbb{Z}$,  $\mathbb{Z}/2\mathbb{Z} \times \mathbb{Z}/4\mathbb{Z}$, or $\mathbb{Z}/6\mathbb{Z}$.  We already showed that $P_1$ does not have order 2, hence the torsion group must be either $\mathbb{Z}/6\mathbb{Z}$  or $\mathbb{Z}/2\mathbb{Z} \times \mathbb{Z}/4\mathbb{Z}$.  We first handle the case of $T=\mathbb{Z}/2\mathbb{Z} \times \mathbb{Z}/4\mathbb{Z}$, showing that when the rank is 0, it cannot have come from a cyclic quadrilateral.

Let the three points of order 2 be denoted $(0,0)$, $T_1$, and $T_2$.  The point $P_1=(x_1,y_1)$ is not a point of order 2, and hence must be a point of order 4.  By Lemma 4.1, $2P_1\neq (0,0)$, and without loss of generality we may take $2P_1=T_1$.  Then the four points which are not of order 2 are $\{P_1, -P_1, P_1+(0,0), -P_1+(0,0)\}$, and it must be that $P_2=(x_2,y_2)$ is one of these points.  If $P_2=P_1$, then $a=y_1/x_1=y_2/x_2=-b$, a contradiction as $a,b>0$.  Similarly, if $P_2=-P_1+(0,0)$ then we again end up with $a=-b$, a contradiction.  If $P_2=P_1+(0,0)$ then $P_3=-(P_1+P_2)=-2P_1+(0,0)=T_1+(0,0)=T_2$, however $P_3=(x_3,y_3)$ is not a point of order 2 since $y_3 \neq 0$.  We can therefore conclude that $P_2=-P_1$, but this is likewise a contradiction because then $P_3=-P_1-P_2=\infty$.  Thus the torsion group for a rank 0 curve $E_{\alpha,-n^2}$ arising from a cyclic quadrilateral cannot be $\mathbb{Z}/2\mathbb{Z} \times \mathbb{Z}/4\mathbb{Z}$.

So we can assume the torsion group is $\mathbb{Z}/6\mathbb{Z}$.
If $P_1$ has order 6, then the set of all rational points of $E$ is $\{P_1,2P_1,3P_1=(0,0),4P_1=P_1+(0,0),5P_1=-P_1,\infty  \}$.  Then
$$P_1+(0,0)=\left(\frac{(a-b-c-d)(a+b-c+d)}{4},-a(a-b-c-d)(a+b-c+d)  \right).$$
As $2P_1=-4P_1$, then
$$2P_1=\left(\frac{(a-b-c-d)(a+b-c+d)}{4},a(a-b-c-d)(a+b-c+d)  \right).$$
Now notice that the ratio of  $y_k/x_k$ for $kP_1=(x_k,y_k)$ (with $k=1,2,4,5$) is either $a$ or $-a$.  But we have $P_2=(x,-bx)$, which must be one of the points $P_1,2P_1, 4P_1, 5P_1$, and hence we must have $a=b$ (since $a,b>0$).  Using the doubling formula, we have that the $x$-coordinate of $2P_1$ is $\frac{b^2(c+d)^2}{4b^2}$, (as $a=b$) which must equal the $x$-coordinate of $P_1+(0,0)$ which is $(-1/4)(c+d)(2b-c+d)$.  Equating these two $x$-coordinates requires $(c+d)(b+d)/2=0$, a contradiction.  Thus, if the rank is 0 then $P_1$ cannot have order 6.

The only other possibility is that $P_1$ has order 3.  Then necessarily $Q=P_1+(0,0)$ has order 6.  The set of all rational points must be $\{Q=P_1+(0,0),2Q=-P_1,(0,0),4Q=P_1,5Q=-P_1+(0,0),\infty \}$.  Similarly as above, the ratio $y/x$ of the points not equal to $(0,0)$ or $\infty$ is equal to $\pm a$.  Considering $P_2$, we must have $b=a$.  This means that $Q=\left((-1/4)(c+d)(2b-c+d),(b/4)(c+d)(2b-c+d)   \right)$.  Using the doubling formula for $2Q$, the $x$-coordinate is $(c+d)^2/4$, which must equal the $x$-coordinate of $-P_1=(1/4)(2b+c-d)(c+d)$.  Equating these two yields $(c+d)(b-d)=0$.  Thus $a=b=d$.  Checking the area, we see the area will be rational if and only if $(a+c)(3a-c)$ is a square.

We observe that since any cyclic quadrilateral with two non-consecutive sides equal is a trapezoid, we have an isosceles trapezoid (which is not a square) with three sides equal if $a \neq c$.  If $a=c$, then we have a square, as a cyclic rhombus must be square.

\end{proof}
So we see that when the rank is 0, the cyclic quadrilateral must have at least three sides equal.  We now examine the converse, which is distinguished by whether or not the quadrilateral is a square.

\begin{theorem}
If the quadrilateral is a square, then the rank is  0, with torsion group $\mathbb{Z}/6\mathbb{Z}$.
\end{theorem}
\begin{proof}
If the quadrilateral is a square, i.e., $a=b=c=d$. Then the curve $E_{\alpha,-n^2} $ is $y^2=x^3+a^2x^2-a^4x$. For any $a\neq 0$ we can do a change of variables $x=a^2X$, $y=a^3Y$, which shows that this curve is isomorphic to $E:Y^2=X^3+X^2-X$, which is a curve of rank $0$, with six torsion points.\\
Thus also the curve $E_{\alpha,-n^2} $ has only 6 rational points. These points are $\{\mathcal{O}, (0,0)$, $(a^2,\pm a^3)$, $(-a^2,\pm a^3)\}.$
The point $P = (-a^2,a^3)$ has exact order 6. The rank of $E_{\alpha,-n^2} $ is zero.
\end{proof}

The rank need not be 0 for isosceles trapezoids with three sides equal.  Take, for example the quadrilateral with side lengths $(13,13,23,13)$, which yields the curve $E_{-11,-216^2}$.  This curve has rank 1, with generating point $(-196,1092)$ and has torsion group $\mathbb{Z}/6\mathbb{Z}$.  We also remark that we can have rank 0 curves $E_{\alpha,-n^2}$ with torsion group $\mathbb{Z}/2\mathbb{Z} \times \mathbb{Z}/4\mathbb{Z}$.  Take for example $\alpha=7, n=12$.  The previous results prove that such curves do not correspond with a cyclic quadrilateral.  If we allow quadrilaterals with $d=0$, then it can be shown that these curves come from quadrilaterals with $d=0$, i.e. triangles with rational area.



\begin{theorem}
If the quadrilateral is a three-sides-equal trapezoid $(a,a,c,a)$, then the torsion group is $\mathbb{Z}/6\mathbb{Z}$.
\end{theorem}
\begin{proof}
If the the quadrilateral is of the form $(a,a,c,a)$, then we have the point $(\frac{1}{4}(a+c)^2,\frac{1}{16}a^2(a+c)^4)$ which can be checked has order 3.  By Corollary 4.5, we immediately have that the torsion group is $\mathbb{Z}/6\mathbb{Z}$.
%
\end{proof}

By scaling the sides of a rational cyclic quadrilateral, we can always assume that area is an integer $N$ (since we are only considering quadrilaterals with rational area).  Using the definition of the generalized congruent number elliptic curve and taking into account Theorems 3.3, 3.4, and 3.5, we can restate Theorem 3.1 in the following way.
\begin{theorem}
Every non-square and non-three-sides equal trapezoidal rational cyclic quadrilateral with area $N \in \mathbb{N}$ gives rise to a generalized congruent number elliptic curve $E_{\alpha,-N^2}$ with positive rank.  Conversely, for any integer $N$ and generalized congruent curve $E_{\alpha,-N^2}$ with positive rank, there are infinitely many non-rectangular cyclic quadrilaterals.
\end{theorem}
\begin{proof}
Given a non-square and non-three-sides equal trapezoidal cyclic quadrilateral with side lengths $(a,b,c,d)$ and area $n=p/q$, consider the quadrilateral with side lengths $(qa,qb,qc,qd)$ which has area $pq \in \mathbb{Z}$.  By our correspondence in Theorem \ref{main}, we can construct the curve $E_{\alpha,-(pq)^2}$, where $\alpha=q^2(a^2+b^2-c^2+d^2)/2$.  If the curve were to have rank 0, then by Theorem \ref{rank0} the quadrilateral would have to have 3 sides equal, which it doesn't.  Hence $E_{\alpha,-(pq)^2}$ has positive rank.

For the converse, given any integer $N$ and $\alpha$ such that $E_{\alpha,-N^2}$ has positive rank, then again by Theorem \ref{main} we are able to construct a cyclic quadrilateral with area $N$.  As the rank is positive, we have an infinite number of choices for $P_1,P_2$ in the correspondence, yielding an infinite number of cyclic quadrilaterals.  If the quadrilateral were to be a rectangle, then $b=d$.  Recall $d=y_{P_3}/x_{P_3}$, where $P_3=-(P_1+P_2)$.  But there are only five points $P=(x,y)$ on $E$ such that $y/x=d$, so we can choose $P_1,P_2$ so as to avoid these five points.
\end{proof}
We conclude this section by noting there are an infinite number of non-rectangular cyclic quadrilaterals with area $n$, for any rational $n$.  Specifically, consider the isosceles trapezoid (which is necessarily cyclic) with side lengths $(j^2+k^2,\ell,j^2+k^2,\ell+2j^2-2k^2)$, where $j>k$.  The height of this trapezoid is $2jk$, yielding an area of $2jk(\ell+j^2-k^2)$.  Hence, by choosing $\ell=\frac{n}{2jk}+k^2-j^2$, the trapezoid will have area $n$.

\section{Torsion points}
In this section, we examine the possible torsion groups $T$ for the curve $E_{\alpha,-n^2}$, corresponding to a cyclic quadrilateral.  By a theorem of Mazur \cite{Sil}, the only possible torsion groups over
$\mathbb Q$, $E(\mathbb Q)_{\rm tors}$, are $\mathbb Z/n\mathbb Z$ for $n=1,2,\ldots, 10, 12$ or
$\mathbb Z/2\mathbb Z\times\mathbb Z/2n\mathbb Z$ for $1\leq n\leq4$.
The point $P_2=(0,0)$ has order 2, hence we know the order of the torsion group must be even.  We will show that the torsion group must be $\mathbb{Z}/2\mathbb{Z}$, $\mathbb{Z}/6\mathbb{Z}$, $\mathbb{Z}/2\mathbb{Z} \times \mathbb{Z}/2\mathbb{Z}$, or $\mathbb{Z}/2\mathbb{Z} \times \mathbb{Z}/4\mathbb{Z}$.  We begin by showing $T \neq \mathbb{Z}/4\mathbb{Z},$ $\mathbb{Z}/8\mathbb{Z}$, or $\mathbb{Z}/12\mathbb{Z}$.

\begin{lemma}
There is no point $P$ on the curve $E_{\alpha,-n^2}$ such that $2P=(0,0)$.  Consequently, the torsion group $T \neq \mathbb{Z}/4\mathbb{Z},$ $\mathbb{Z}/8\mathbb{Z}$, or $\mathbb{Z}/12\mathbb{Z}$.
\end{lemma}
\begin{proof}
Let $P=(x,y)$ be a point on $E_{\alpha,-n^2}$ such that $2P=(0,0)$.  By using the formula for doubling a point (see \eqref{double}), we must have
$$\frac{(x^2+n^2)^2}{4y^2}=0.$$
However, this is clearly impossible, since $x^2+n^2>0$.  Thus no such point $P$ exists.

Note that if the torsion group were $\mathbb{Z}/4\mathbb{Z},$ $\mathbb{Z}/8\mathbb{Z}$, or $\mathbb{Z}/12\mathbb{Z}$, then there would necessarily have to be a point $P$ with $2P=(0,0)$, since $(0,0)$ would be the unique point of order 2.  As this is not possible, then $T$ cannot be any of these three groups.
\end{proof}

\begin{proposition}
There are no points of order 5 on the curve $E_{\alpha,-n^2}$.
\end{proposition}
\begin{proof}
By an old result, an elliptic curve with a rational point of order 5 will have its $j$-invariant of the form $(s^2+10s+5)^3/s$ for some $s \in \mathbb Q$ \cite{Fricke}.
Calculating the $j$-invariant of $E_{\alpha,-n^2}$, we must therefore have $$256\frac{(\alpha^2+3n^2)^3}{n^4(\alpha^2+4n^2)}=\frac{(s^2+10s+5)^3}{s}.$$
Let $w=\alpha^2+4n^2$, so that
$$256\frac{(w-n^2)^3}{n^4w}=\frac{(s^2+10s+5)^3}{s},$$
and write $w-n^2=c(s^2+10s+5)$.  Simplifying we obtain the (genus 0) curve
$$256c^3s-n^6-n^4cs^2-10n^4cs-5n^4c=0,$$
with rational point $(c,s)=(-n^2/32,9/4)$.
We can parameterize all solutions by
$c=-\frac{1}{32}\frac{n^2(144n^8m^2+24n^4m+1)}{6n^4m+1},$ and $s=\frac{3}{4}\frac{20n^4m+3}{432n^{12}m^3+180n^8m^2+24n^4m+1}$.
Since $\alpha^2=w-4n^2=c(s^2+10s+5)-3n^2$,
we can solve for $\alpha^2$ in terms of $m$ (and $n$):
$$\alpha^2=-\frac{n^2(720n^8m^2+216n^4m+17)(144n^8m^2+96n^4m+11)^2}{512(6n^4m+1)^5}.$$
This equation will have rational solutions if and only if the curve
$$C:z^2=-2(6n^4m+1)(720n^8m^2+216n^4m+17)$$
has rational points.  The curve $C$ is birationally equivalent to the curve
$$E:Y^2=X^3+76032X-8183808,$$
using the maps
$$(X,Y)=\left(-8640n^4m-1344, 8640z \right),$$
$$(m,z)=\left(-\frac{X+1344}{8640n^4}, \frac{Y}{8640} \right).$$
With SAGE, we compute that this curve $E$ has rank 0, with only one torsion point (96,0) \cite{Sag}.  This corresponds to having $(m,z)=(-1/(6n^4),0)$, which is thus the only rational point on $C$.  However, this leads to no rational points on the curve relating $\alpha^2$ and $n,m$.  From this we may conclude our initial assumption was not possible.  Thus, there are no elliptic curves resulting from cyclic quadrilaterals with rational area which have a point of order 5.
\end{proof}

The previous two lemmas show that if $T$ is cyclic, then $T$ must either be $\mathbb{Z}/2\mathbb{Z}$ or $\mathbb{Z}/6\mathbb{Z}$.  We now turn to the case when $T$ is not cyclic, i.e., $T$ has more than one point of order 2.  This will occur precisely when $x^2+\alpha x-n^2=0$ has a rational root, which happens if and only if the discriminant $\alpha^2+4n^2=a^2b^2+a^2d^2+b^2d^2+2abcd$ is a square.

\begin{lemma}
The torsion group $T$ is not $\mathbb{Z}/2\mathbb{Z} \times \mathbb{Z}/8\mathbb{Z}$.
\end{lemma}
\begin{proof}
Suppose we have a curve with torsion group $\mathbb{Z}/2\mathbb{Z} \times \mathbb{Z}/8\mathbb{Z}$.  Then necessarily we have three 2-torsion points, so
we can assume $x^2+\alpha x-n^2=(x+M)(x+N)$, for some rational $M,N \neq 0$.  The points $(0,0), (-M,0),$ and $(-N,0)$ all have order 2 on $E_{\alpha,-n^2}$.
By a theorem of Ono [\cite{Ono},Main Theorem 1], the torsion group of $E_{\alpha,-n^2}$ contains $\mathbb{Z}/2\mathbb{Z} \times \mathbb{Z}/4\mathbb{Z}$ as a subgroup if (i) $M$ and $N$ are both squares, or if (ii) $-M$ and $N-M$ are both squares, or if (iii) $-N$ and $M-N$ are both squares.  We show only case (iii) is possible.

Without loss of generality we may take $\alpha^2+4n^2=r^2$, (with $r>0$) and $M=\frac{1}{2}(\alpha+r)$, $N=\frac{1}{2}(\alpha-r)$.  For case (i), if $M$ and $N$ are both squares, then so is $MN$.  However, this is a contradiction as $MN=-n^2$ and $n>0$.  For case (ii), $N-M=-r$ and so cannot be a square ($r>0$).

Therefore we must be in case (iii), and we have both $-N$ and $M-N=r$ squares.  If we write $-N=j^2$ for some $j$, then $MN=-n^2$ and so $M=-n^2/N=n^2/j^2$, hence $M$ is square.  By the second part of Ono's theorem, if $T= \mathbb{Z}/2\mathbb{Z} \times \mathbb{Z}/8\mathbb{Z}$ then $M=u^4-v^4$ and $N=-v^4$, with $u^2+v^2=w^2$.  As $M$ is square then $u^4-v^4=z^2$ for some rational $z$.  This equation is well known to have no non-trivial solutions, i.e. only when $M=u^4-v^4=0$.  However, this contradicts our initial assumption, and so
supposing $T=\mathbb{Z}/2\mathbb{Z} \times \mathbb{Z}/8\mathbb{Z}$ must have been incorrect.

\end{proof}

\begin{lemma}
The torsion group $T$ is not $\mathbb{Z}/2\mathbb{Z} \times \mathbb{Z}/6\mathbb{Z}$.
\end{lemma}
\begin{proof}
It is well known that a curve has a point of order 3 if and only if the 3-torsion polynomial $\psi_3$ has a root.  For $E_{\alpha,-n^2}$, this is
$$\psi_3(x)=3x^4+4\alpha x^3-6n^2x^2-n^4=0,$$
or equivalently
\begin{equation}
\label{z2z6}
3x^4+4w x^3-6x^2-1=0,
\end{equation}
where $w=\alpha /n.$  In order to have more than one 2-torsion point, we also must have that $\alpha^2+4n^2$ is a square, or equivalently $w^2+4$ is a square.  We may parameterize to find that $w$ can be written $w=\frac{4-j^2}{2j},$ for some rational $j$.  Substituting this back into \eqref{z2z6}, we see that
$$C:3jx^4+8x^3-2j^2x^3-6jx^2-j=0.$$
This is a genus one curve, birationally equivalent to the curve
$$E:Y^2=X^3-13/3X-70/27$$
via the maps
$$(X,Y)=\left(-\frac{2x^4-6jx^3-13x^2-3}{3x^2(x^2+1)}, -\frac{x^4+2jx^3+6x^2+1}{x^3(x^2+1)}   \right),$$
$$(x,j)=\left(-\frac{9Y}{9X^2-15X-14} ,-\frac{216Y}{27X^3+27X^2-135X-175} \right).$$
Using SAGE, we compute the curve $E$ has rank 0, and only the three 2-torsion points (7/3,0),(-2/3,0),(-5/3,0) \cite{Sag}.  Tracing these points back through the substitutions, we get no rational points on the curve $C$ besides $(0,0)$.  Thus, this torsion group is not possible.
\end{proof}

Combining the above series of results, we immediately obtain the following corollary.
\begin{corollary}
Given a cyclic quadrilateral with corresponding elliptic curve $E_{\alpha,-n^2}$, the torsion group must be $\mathbb{Z}/2\mathbb{Z}$, $\mathbb{Z}/6\mathbb{Z}$,  $\mathbb{Z}/2\mathbb{Z} \times \mathbb{Z}/2\mathbb{Z}$, or $\mathbb{Z}/2\mathbb{Z} \times \mathbb{Z}/4\mathbb{Z}$.
\end{corollary}
\begin{proof}
By Mazur's theorem, we have a finite list of possible torsion groups.  We know the torsion group must have order divisible by 2, as the point $(0,0)$ has order 2.  Eliminating the various groups from the previous four lemmas, we have the result.
\end{proof}
We note that all four torsion groups are possible.  As we previously showed, any square will have torsion group $\mathbb{Z}/6\mathbb{Z}$.  For any $m>2$, if we let $a=m^2-4$ and $b=2m$, then the rectangle with side lengths $a$ and $b$ will have torsion group $\mathbb{Z}/2\mathbb{Z} \times \mathbb{Z}/2\mathbb{Z}$.  This follows since the curve is $y^2=x(x+m^4-4m^2)(x-4m^2+16)$, and hence has three points of order 2.  For an example of curves with torsion group  $\mathbb{Z}/2\mathbb{Z} \times \mathbb{Z}/4\mathbb{Z}$, let $\alpha=u^4-6u^2v^2+v^4$ and $n=2uv(u^2-v^2)$.  Then the curve $E_{\alpha,-n^2}$ will have torsion group  $\mathbb{Z}/2\mathbb{Z} \times \mathbb{Z}/4\mathbb{Z}$  with the point $(2uv(u+v)^2,2uv(u^2+v^2)(u+v)^2)$ having order 4.  The most common case for a cyclic quadrilateral is $\mathbb{Z}/2\mathbb{Z}$.

It turns out that if the cyclic quadrilateral is a non-square rectangle, we can rule out two of the possible torsion groups.

\begin{proposition}
The elliptic curve arising from a rectangle has a point of order 3 if and only if the rectangle is actually a square.
\end{proposition}
\begin{proof}
By Theorem 3.4, if we begin with a square, the torsion group is $\mathbb{Z}/6\mathbb{Z}$.  For the converse, suppose we have a curve (arising from a rectangle) which has a point of order 3.  The curve equation is $y^2=x^3+b^2x^2-a^2b^2x$.  Under the isomorphism $(x,y) \to (b^2x ,b^3 y)$, this is the curve $y^2=x^3+x^2-Cx$, where $C=(a/b)^2$.  The three-torsion polynomial for this curve is
$$\Psi_3=3x^4+4x^3-6Cx^2-C^2,$$
and so under our assumption of having a point of order 3, there exists a rational $x$ satisfying $\Psi_3(x)=0$.
Solving for $C$, in terms of $x$, we find
$$C=x\left(-3x \pm 2\sqrt{3x^2+x}  \right).$$
We can parameterize rational solutions of $3x^2+x$ being a square by $x=(m-2)^2/(m^2-3)$.  This makes $C$ either
$$C_1=-\frac{(7m-12)(m-2)^3}{(m^2-3)^2},$$
or
$$C_2=\frac{m(m-2)^3}{(m^2-3)^2}.$$
Now, substituting these values into $x^3+x^2-Cx$, we obtain that both $(3m-5)^2(m-2)^4/(m^2-3)^3$ and $(m-1)^2(m-2)^4/(m^2-3)^3$ must be squares, or equivalently $m^2-3$ must be a square.  We can parameterize the rational solutions of $m^2-3$ being square by
$$m=2\frac{t^2-t+1}{t^2-1}.$$
We substitute this value of $m$ in for $ C_1,$ and $C_2$, obtaining
$$C_1=16\frac{(t^2-7t+13)(t-2)^3}{(t^2-4t+1)^4},$$
$$C_2=-16\frac{(t-2)^3(t^2-t+1)}{(t^2-4t+1)^4}.$$
Recall that $C=(a/b)^2$ is a square, and hence both $C_1$ and $C_2$ must be as well, leading to the equations
$$z_1^2=(t-2)(t^2-7t+13),$$
$$z_2^2=-(t-2)(t^2-t+1).$$
These are both elliptic curves, for which it can be checked that each has rank 0, and exactly 6 torsion points \cite{Sag}.  These are $(t,z_1)=(2,0),(3,\pm 1),(5,\pm 3)$ and $(t,z_2)=(2,0),(1,\pm 1),(-1,\pm 3)$.  Substituting these values of $t$ in and calculating $C_1$ and $C_2$, we find that they lead to $C_1=1$ or 0, and the same holds true for $C_2$.  Thus $C=0$ or $1$, which means that $a/b=0$ or $a/b=1$.  As we are assuming $ab\neq 0$, then $a=b$ and we have a square.
\end{proof}

\begin{corollary}
Given a non-square rectangle with sides $a$ and $b$, then $T=\mathbb{Z}/2\mathbb{Z}$ or  $T=\mathbb{Z}/2\mathbb{Z} \times \mathbb{Z}/2\mathbb{Z}$.  We get the latter group if and only if $4a^2+b^2$ is a square.
\end{corollary}
\begin{proof}
The only torsion group we need to show is not possible is $T=\mathbb{Z}/2\mathbb{Z} \times \mathbb{Z}/4\mathbb{Z}$.  Assume that $E_{\alpha,-n^2}(\mathbb{Q})_{tors} $ contains $T$. As there are three points of order 2, we must necessarily have $4a^2+b^2$ is equal to a square, say $r^2$.  Then a theorem of Ono [\cite{Ono}, Main Theorem 1], implies that \\
$i)$ $b(\frac{-b+ r}{2})$ and $b(\frac{-b- r}{2})$ are both squares, or\\
$ii)$ $-b(\frac{-b+r}{2})$ and $-br$ are both squares, or\\
$iii)$ $-b(\frac{-b-r}{2})$ and $br$ are both squares.\\
For $(i)$, $$(\frac{-b+ r}{2})b=r_1^2,\qquad (\frac{-b- r}{2})b=r_2^2.$$
  Therefore we have $-b^2=r_1^2+r_2^2$ which is contradiction.
For $ii)$, as $b>0$ and we can take $r>0$, then $-br$ is not a square.  For $iii)$, if $br=b\sqrt{b^2+4a^2}=s^2$, then $b^4+4a^2b^2=s^4$.  We may rewrite this as $(2a/b)^2=(s/b)^4-1$, since $b>0$. The only rational points on the curve are $(\pm 1,0)$, and as we can assume $s>0$, then it must be that $s/b=1$.  However, if $s=b$, then we have $4a^2b^2=0$, a contradiction.

The corollary now follows immediately from the previous results in this section.
\end{proof}

\section{Congruent numbers }
Recall a congruent number is an integer $n$ which is the area of a right triangle with rational sides.  It is well known that $n$ is congruent if and only if the elliptic curve $y^2=x^3-n^2x$ has a rational point $P$, which is not of order 2.  These congruent curves are a subset of our curves $E_{\alpha,-n^2}$, with $\alpha=0$.  If the point $P$ has infinite order, then by the main result of this paper, we can construct a cyclic quadrilateral with area $n$, and side lengths $(a,b,c,d)$ such that $a^2+b^2+d^2=c^2$.  Note by setting $d=0$, we get the correspondence between congruent numbers and elliptic curves just mentioned.  In this sense, Theorem \ref{main} provides a generalization of this congruent number -- elliptic curve connection.  This gives us the following corollary.

\begin{corollary}
An integer $n$ is congruent if and only if there is a cyclic quadrilateral with area $n$, and rational side lengths $(a,b,c,d)$ with $a^2+b^2+d^2=c^2$.
\end{corollary}

This characterization of congruent numbers can be added to the list of the many other known characterizations of congruent numbers. Several of these are given in Koblitz's book \cite{Koblitz}.  Given an integer $n$, it is a well-known open problem to determine whether or not $n$ is congruent. A partial answer is given by Tunnell's theorem, which gives an easily testable criterion for determining if a number is congruent.  However, this result relies on the unproven Birch and Swinnerton-Dyer Conjecture for curves of the form $y^2=x^3-n^2x$.   The criterion involves counting the number of integral solutions $(x,y,z)$ to a few Diophantine equations of the form $ax^2+by^2+cz^2=n$.  See \cite{Tunnell} for more details.

In the remainder of this section, we give infinite families of congruent number elliptic curves with (at least) rank three.  Searching for families of congruent curves with high rank has been done before \cite{Duj5},\cite{Duj6},\cite{Rog},\cite{Rub},\cite{Wat}.  Currently, the best known results are a few infinite families with rank at least 3 \cite{Spe3},\cite{Rub}, and several individual curves with rank 7 \cite{Wat}.


\subsection{A family of congruent number elliptic curves with rank at least 3.}
In order for $n$ to be congruent, we need $\alpha=a^2+b^2-c^2+d^2=0$.  It is known (\cite{And}, Page 79), that we can parameterize the solutions of $a^2+b^2-c^2+d^2=0$ by $a=p^2+q^2-r^2$, $b=2pr$, $c=p^2+q^2+r^2$, and $d=2qr$.  By scaling, we can assume that $r=1$.  The condition that the quadrilateral have rational area is then that
$$(p+q+1)(p^2+q^2-p+q)(p+q-1)(p^2+q^2+p-q)$$
is a square.
To simplify a bit, we require both
$$(p+q-1)(p+q+1)=(p+q)^2-1$$
$$(p^2+q^2-p+q)(p^2+q^2+p-q)=(p^2+q^2)^2-(p-q)^2$$
to be squares.
From the first of these equations, we can parameterize the solutions by $p+q=\frac{1}{2}\frac{z^2+1}{z}.$  We note the bottom expression will be a square if $p=q$, thus we have $p=q=\frac{z^2+1}{4z}$.  We can scale the resulting sides by $4\frac{z^2}{z^2+1}$ so that the area of the resulting quadrilateral is $n=z(z-1)(z+1).$  From our correspondence, we have the following points on the curve $y^2=x^3-n^2x$:
$$P_1=\left(\frac{1}{4}(z^2+1)^2,\frac{1}{8}(z^2+1)(z^2+2z-1)(z^2-2z-1)  \right),$$
$$P_2=\left(-z(z-1)^2,2z^2(z-1)^2  \right).$$
It can be checked that $P_1=-2P_2$, hence we only have a rank 1 family.   A natural approach to finding more rational points on this curve is to look for factors $B$ of $n(z)$ such that $B-n(z)^2/B$ is square.  Note that if $x_3=2z^2(z+1)$, then $x_3-n^2/x_3=\frac{1}{2}(3z-1)(z+1)^2$.  Thus we can set $z=\frac{1}{3}(2t^2+1)$ to obtain a square.  It can be checked (by specializing) that our new point with $x$-coordinate $2z^2(z+1)=\frac{4}{27}(t^2+2)(2t^2+1)$ is linearly independent from $P_1$ (or $P_2$).

We now repeat the process, with (after scaling)
$$n=3(t-1)(t+1)(t^2+2)(2t^2+1).$$
With $x_4=-3(t-1)(t+1)(t^2+2)^2$, then we will obtain another rational point if $t^2+1$ is a square.  We again parameterize by setting $t=\frac{w^2-1}{2w}$.  Checking the factors of $n$, we did not find any new independent points on the curve.

We summarize the above results.  The cyclic quadrilateral with side lengths
$$a=\frac{1}{2}\frac{(w^8-12w^6-34w^4-12w^2+1)(w^8+12w^6-34w^4+12w^2+1}{w(w^8+38w^4+1)},$$
$$b=d=12w(w^4+1),$$
$$c=\frac{1}{2}\frac{w^{16}+364w^{12}+2022w^8+364w^4+1}{w(w^8+38w^4+1)},$$
has area
$$n=6(w^2+2w-1)(w^2-2w-1)(w^4+1)(w^4+6w^2+1).$$
The resulting congruent curve $y^2=x^3-n^2x$ has three independent points with the following $x$-coordinates
$$x_1=-6(w^4+1)(w^2+2w-1)^2(w^2-2w-1)^2,$$
$$x_2=\frac{3}{16}\frac{(w^4+1)^2(w^4+6w^2+1)}{w^6},$$
$$x_3=-\frac{3}{64}\frac{(w^2+2w-1)(w^2-2w-1)(w^4+6w^2+1)^2}{w^6}.$$
The points being independent can be checked by specialization -- for instance, when $w=2$, the height pairing matrix has determinant 43.6831845338168 as computed by SAGE.
This family has been previously discovered in \cite{Rub}.  We also note that for $w=14/9$ (or $w=5/23,23/5,9/14$ which all yield the same curve) we obtain a rank 6 curve.

\subsection{Other families of rank 3 congruent number curves}
We can find other families with rank 3 using the same technique illustrated in the previous subsection.  For example, instead of selecting the particular $x_4$, if we had instead chosen
$x_4=6(t-1)(t+1)(2t^2+1)^2$, then $x_4$ will lead to a rational point if $10t^4-2t^2-8$ is a square.  The equation $C:s^2=10t^4-2t^2-8$ has the rational point $(2,12)$, and hence is an elliptic curve.  There is a birational transformation from $C$ to the curve $E:y^2=x^3+58x^2+1440x+12960$, given by $t=-(x+36)/x, s=36y/x^2$.  The curve $E$ has rank 1, with generator $P=(-12,48)$.  As we have an infinite number of points on $E$, then we get an infinite number of congruent number curves with 3 independent points.  Specifically, given $(x,y)$ on $C$, let $t=-(x+36)/x$, then $x_4$ as defined above will give a rational point on the congruent number curve.

If we start with other parameterizations for $(a,b,c,d)$, it is not hard to find other families of congruent number elliptic curves with rank (at least) 3 using the same techniques.  As a final example, let
$$a=(t+1)(t-1)(1+5t+t^2)(1-5t+t^2),$$
$$b=-\frac{1}{3}(-2+t))(-1+2t)(1+2t)(2+t)(t-1)(t+1),$$
$$c=-\frac{13-61t^2+177t^4-61t^6+13t^8}{3(t-1)(t+1)},$$
$$d=-\frac{(1+t+t^2)(1-t+t^2)(2+t)(1+2t)(-1+2t)(-2+t)}{(t-1)(t+1)}.$$
The area of the cyclic quadrilateral is then
$$n=2t(2+t)(1+2t)(-1+2t)(-2+t)(t^2+1)(t^4+7t^2+1).$$
The points arising from the $x$-coordinates
$$x_1=(1+t^2)^2(t^4+7t^2+1)^2$$
$$x_2=-(9(t+2))(1+2t)(1-2t)(2-t)t^2(1+t^2)^2$$
are linearly independent.  If we set
$$x_3=-2t(1-2t)^2(2-t)^2(1+t^2)(t^4+7t^2+1)$$
then this will be a point provided that $5t^4+35t^2+5$ is a square.  This is birationally equivalent to the elliptic curve $E:y^2-20xy-1200y=x^3+55x^2-4500x-247500$, which has rank 1.  Specifically, given a point $(x,y)$, let $t=(30x+1650)/y-2$, leading us to obtain a third point.  Specializing, we see we get an infinite family of rank 3 congruent number curves, arising from the infinite number of points on $E$.

Both of the examples given in this subsection are new, meaning they have not appeared in the literature before.  We find it not too difficult to generate a large number of rank 2 families, which will have a third independent point arising from an associated elliptic curve with positive rank.  We did not find any unconditional rank 3 family besides the one given in the previous subsection, nor did we find a family with rank 4.  These families could be useful in a new search for congruent curves with high rank.  The search performed in \cite{Wat} used a rank 2 family, and it is possible that new families could lead to more congruent curves with high rank.

\section{High rank curves with torsion group $\mathbb{Z}/2\mathbb{Z}$}
In this section, we look for infinite families of curves corresponding to cyclic quadrilaterals with high rank, as well as specific curves in these families with high rank.  The family of curves $E_{\alpha,-n^2}$ is a subset of the more general family of elliptic curves with a 2-torsion point.  Studying families of elliptic curves with torsion group $\mathbb{Z}/2\mathbb{Z}$ with high rank has been of much interest, as described in the introduction. The highest known rank for a curve with $T=\mathbb Z/2\mathbb Z$ is 19, due to Elkies \cite{Duj3}.  For infinite families, Fermigier found some with rank at least 8.  We use our correspondence to find an infinite family with rank 5, and specific curves with ranks as high as 10 (see Table 1 in Section 7).

\subsection{An infinite family with rank at least 4}

A quadrilateral with side lengths $(a,c+2,c,a+2)$ will have rational area if $a(a+2)c(c+2)$ is square.  We parameterize the solutions of $a(a+2)$ being a square by $a=-u^2/(2u-2)$, and similarly set $c=v^2/(2v+2)$.  By using the same technique as described in Section 5.1, we find that $x_3=u^2v(u-1)(v+1)(v+2)$ will be the $x$-coordinate of a rational point if $v=-\frac{1}{8}\frac{u^4-4u^3+20u^2-32u+16-z^2}{(u-1)^2}.$  We clear denominators, and repeat the procedure to obtain the following family of rank (at least) 4.

Set
$$ \alpha=u^4w^2+4w^4-8w^4u+4w^4u^2+8w^2-16w^2u+12w^2u^2+4-8u+4u^2-4w^2u^3,$$
$$ n^2=(w^4-1)^2u^2(u-2)^2(u-1)^2.$$
 Let the curve $E_{u,w}$ be defined with this value of $\alpha$ and $n^2$.  Then
 $$ x_1=-(w-1)^2(w+1)^2u^2(u-1),$$
 $$x_2=-u^2w^2(u-2)^2,$$
 $$ x_3=-4(w^2+1)^2(u-1)^2,$$
are all rational $x$-coordinates of points on $E_{u,w}$.  If we further set
$w=(m^2+2(u-1)(u^2-2u+4)m+8(u-1)^2)/((u-1)m^2-8(u-1)^3)$, then
 $$x_4=(w^4-1)u^2(u-1)$$
will also be a rational $x$-coordinate of a point on the curve. We performed a computer search to find elliptic curves in the family $E_{u,w}$ with high rank, and found hundreds of curves with rank 9 and several curves with rank 10 (see Table 1).  We remark that other families could be similarly constructed if we begin with different side lengths.  Also note that the curves $E_{u,-w}$ and $E_{u,1/w}$ both are equivalent to the curve $E_{u,w}$.  


\subsection{A family  with rank at least 5}
We conclude this section with a subfamily with rank 5.  
We search for a fifth point $P_5$, which will be rational if a certain quartic equation in $m$ is a square.
Specifically, let $x_5=-x_4=-u^2(u-1)(w^4-1)$.  To yield a rational point, we need that  
the quartic (in $m$)
$$\begin{aligned}
(3&u^2-6u+4)^2m^4+4(u-1)(u^2-2u+4)(u^4-4u^3+12u^2-16u+8)m^3\\
&+4(u-1)^2(u^8-8u^7+40u^6-128u^5+268u^4-368u^3+400u^2-320u+128)m^2\\
&+32(u-1)^3(u^2-2u+4)(u^4-4u^3+12u^2-16u+8)m+64(u-1)^4(3u^2-6u+4)^2, 
\end{aligned}$$
is square.
As the coefficient of $m^4$ is square, we may use a technique attributed to Fermat (\cite{Dic}, p. 639) to solve for $m$ in terms of $u$ so that the resulting equation is square.  A short calculation finds
$$m=-\frac{4(u-1)(u^8-8u^7+34u^6-92u^5+178u^4-248u^3+232u^2-128u+32)}{(u^2-2u+4)(3u^2-6u+4)^2}.$$
Thus, we have a parameterized family with 5 rational points.  Specializing (at $u=3$ for example), shows the five points are linearly independent, and hence the rank of this family is at least five.

We performed a computer search for high rank curves in the rank 5 family, but the search was not nearly as successful as for the $E_{u,w}$ family since the size of the coefficients were so large.  Note that this rank 5 family is a subset of the $E_{u,w}$ family anyhow.

\section{Examples and Data}
Our starting point for examples of high rank curve is the family of elliptic curves with rank at least 4 from Section
6.1. We use the sieving method based on Mestre-Nagao sums
$$S(N,E) =\sum_{p\leq N, p     \medspace  prime}(1-\frac{p-1}{ |
E(\mathbb{F}_p)|})\log (p).$$
(see \cite{Mes,Na}). For curves with large values of $S(N, E)$, 
we compute the Selmer rank, which is a well-known upper bound for the rank. Specifically we searched for curves $E$ that satisfied the bounds $S(523,E)>20$ and $S(1979,E)>28$.  We
combine this information with the conjectural parity for the rank.\\
Finally, we try to compute the rank and find generators for the best candidates
for large rank. We have implemented this procedure in SAGE \cite{Sag} and PARI \cite{Pari}, using Cremona's
program \texttt{mwrank} \cite{Mw}, for the computation of rank and Selmer rank.  
In the following table, we present examples of the curves we found with rank 10.   

Finally, in Table 2 we present examples of cyclic quadrilaterals with rational area $n$ and associated elliptic curves for positive integers $n$ up to 50.  For a given $n$, there are many cyclic quadrilaterals that we could use, however we chose ones which had numerators and denominators relatively small.  Note the rank for all these curves is at least 2.\\

\begin{table}[h]
\caption{High rank curves in the family $E_{u,w}$ from Section 6.1}  
\centering          
\begin{tabular}{c c c }    
\hline\hline                        
$u$ &$w$ & rank \\ [0.5ex]  
\hline                      
-84/11 & 29/14 & 10 \\
-63/22 & 97/5 & 10 \\
-62/81 & 32/9 & 10 \\
-60/77 & 22/3 & 10 \\
-53/77 & 31/5 & 10 \\
-47/27 & 45/7 & 10 \\
-32/77 & 49/25 & 10 \\
7/11 & 3161/4679 & 10 \\
9/25 & 6091/19600 & 10  \\        
63/85 & 5/97 & 10 \\
\hline          
\end{tabular}
\label{table:nonlin}    
\end{table}

\begin{table}[!htb]
\centering
\caption { Transformation from $E_{\alpha,-n^2}$ to cyclic quadrilateral.}
\label{Tab:1}
\small
{
\begin{tabular}{c  c  c c }
\hline \hline
$n$&$\alpha$&rank&$[a,b,c,d]$\\
\hline
1&5/2& 2& [5/6, 1, 5/6, 2]\\
\hline
2&1/9& 2& [1/3, 4/3, 8/3, 7/3]\\
\hline
3&3& 2 &[3, 1/2, 4, 3/2]\\
\hline
4&10& 2 &[5/3, 2, 5/3, 4]\\
\hline
5&1/9& 2 &[1/3, 7/3, 13/3, 11/3]\\
\hline
6&9/40&3&[1,12/5, 2813/680, 447/136]\\
\hline
7& -7/18 & 2 &[4/3, 7/2, 9/2, 7/3]\\
\hline
8&1/4 &2& [1/2, 7/2, 31/6, 23/6]\\
\hline
9&-5/4&2& [2, 5/2, 5, 7/2]\\
\hline
10&1& 2& [1, 4, 16/3, 11/3]\\
\hline
11&1/4& 2& [1/2, 13/6, 127/15, 41/5]\\
\hline
12 &46 & 2 & [1,6,3,8]\\
\hline
13&829/4& 2& [1/4, 16, 13/4, 13]\\
\hline
14&1/4 &2 &[1/2, 2, 34/3, 67/6]\\
\hline
15&15& 2 &[4, 3/2, 5, 15/2]\\
\hline
16 & 40 & 2 & [10/3, 4, 10/3, 8]\\
\hline
17&4 &2& [2, 106/39, 8017/1092, 199/28]\\
\hline
18&1& 2 &[1, 4, 8, 7]\\
\hline
19&1 &2 &[1, 17/2, 91/10, 17/5]\\
\hline
20&7 &2 &[4, 7, 22/3, 5/3]\\
\hline
21&1/4& 2 &[1/2, 15/2, 9, 5]\\
\hline
22& -149/3 & 2 &[3,20/3,40/3,5]\\
\hline
23&4 &3 &[2/5, 2, 383/20, 77/4]\\
\hline
24&9/10& 3& [2, 24/5, 2813/340, 447/68]\\
\hline
25&4/9& 2 &[2/3, 58/3, 233/12, 23/12]\\
\hline
26&64&2&[8, 5, 7/3, 20/3]\\
\hline
27&8/5 &2 &[12/5, 176/35, 597/70, 67/10]\\
\hline
28&-14/9& 2 &[8/3, 7, 9, 14/3]\\
\hline
29&-1/5& 2& [437/120, 569/56, 682639/62160, 1139/592]\\
\hline
30&4& 2 &[2, 6, 9, 7]\\
\hline
31&4/9 & 2 &[2/3, 1069/24, 32427/728, 66/91]\\
\hline
32&1 &2& [1, 7, 31/3, 23/3]\\
\hline
33&-7/4 &2 &[2, 41/10, 58/5, 21/2]\\
\hline
34&2 &2 &[19/6, 311/36, 352771/35460, 25291/5910]\\
\hline
35&385/9 &2& [7/3, 35/3, 9, 5]\\
\hline
36&-5& 2& [4, 5, 10, 7]\\
\hline
37&-4/3& 2& [5/6, 9/2, 175/12, 55/4]\\
\hline
38&4/9 &2 &[2/3, 17/3, 197/15, 178/15]\\
\hline
39&333/2 &2& [9/2, 15, 7/2, 10]\\
\hline
40&-3 &2 &[7/3, 6, 34/3, 9]\\
\hline
41&4/9& 3 &[2/3, 98/15, 1427/110, 247/22]\\
\hline
42&1 &3 &[1,9,12,8]\\
\hline
43&-8/9& 3 &[104/15, 32/3, 5387/420, 79/84]\\
\hline
44&-2/3 &3 &[10/21, 5, 353/21, 16]\\
\hline
45&1&2& [1,7,13,11]\\
\hline
46&1/4 &2 &[1/2, 65/2, 2867/88, 205/88]\\
\hline
47&4 &2 &[2/7, 2, 1727/42, 247/6]\\
\hline
48&184 & 2& [2, 16, 6, 12]\\
\hline
49&245/6&2&[35/6, 7, 35/6, 14]\\
\hline
50&25/9 &2 &[5/3, 20/3, 40/3, 35/3]\\
\bottomrule
\end{tabular}
}
\end{table}

\clearpage

{{\bf Acknowledgements:} The authors would like to thank Peter Mueller for the idea behind the proof of Proposition 4.2.}

\bibliographystyle{amsplain}

\end{document}